\documentclass{article}
\usepackage{graphicx}
\usepackage[nonatbib, preprint]{neurips_2025}
\usepackage{amsmath, amsfonts, amssymb}
\usepackage{algorithm}
\usepackage[noend]{algpseudocode}
\usepackage{makecell}
\usepackage{wrapfig}
\usepackage{enumitem}
\usepackage{booktabs}
\usepackage{multirow}
\usepackage{fancyhdr}
\usepackage{comment}
\usepackage{bbm}
\usepackage{caption}
\usepackage{subcaption}
\usepackage{siunitx}
\usepackage{xcolor}
\usepackage{amsthm}
\usepackage{dsfont}
\usepackage{hyperref}
\usepackage{biblatex}
\usepackage{array}
\bibliography{ref}
\usepackage{siunitx}
\usepackage{listings}
\usepackage{float}
\newfloat{listing}{h}{lop}
\floatname{listing}{Listing}
\lstdefinestyle{pystyle}{
  language=Python,
  basicstyle=\ttfamily\small,
  keywordstyle=\color{blue!70!black}\bfseries,
  commentstyle=\color{gray},
  stringstyle=\color{teal},
  numbers=none,
  showstringspaces=false,
  breaklines=false,
  escapeinside={(*@}{@*)},
}
\numberwithin{equation}{section}
\newtheorem{theorem}{Theorem}
\newtheorem{proposition}{Proposition}
\newtheorem{lemma}{Lemma}

\DeclareMathOperator*{\argmax}{arg\,max}

\usepackage{xspace}

\newcommand{\interior}{\operatorname{int}}
\newcommand{\unitvector}{\mathds{1}_n}

\title{Refined outer-approximation algorithms for monotonic optimisation}
\author{
  Ahmed Rashwan\thanks{\texttt{ar3009@bath.ac.uk}} \\
  University of Bath
}
\date{}

\begin{document}

\maketitle
\setcounter{footnote}{0}

\begin{abstract}
Monotonic optimisation is a broad class of non-convex problems formulated in terms of monotone functions. Such problems are commonly solved via the \emph{polyblock outer-approximation algorithm (POA)}, a branch-and-bound method that iteratively refines a rectangular outer-approx\-imation of the feasible set.
POA scales poorly, however: the number of vertices needed to describe the approximation can grow exponentially, leading to large memory requirements and increasingly expensive subroutines.
To address these limitations, we propose three algorithmic improvements: a generalised anchor selection that yields an optimal balanced monotonicity cut, a relaxed optimality condition that guarantees finite termination without continuity assumptions, and a vectorised variant that processes multiple nodes concurrently.
We further develop an efficient tree-based implementation, which accelerates POA's core subroutines while storing the outer-approximation compactly.
Numerical experiments show that these improvements yield order-of-magnitude speed-ups over standard POA and solve problems on which existing variants fail.
Finally, we introduce \texttt{polyblocks}\footnote{Available at \url{https://github.com/RashwanA/polyblocks}}, an open-source Python package implementing the proposed algorithms alongside a framework for developing new POA variants.
\end{abstract}

\section{Introduction}
An optimisation problem is said to be \emph{monotonic} if it is formulated in terms of non-decreasing (or non-increasing) functions.
Monotonicity is a common, and often easy to identify, structure in optimisation which can be exploited to find globally optimal solutions for otherwise difficult non-convex problems.
Monotonic optimisation consequently forms one of the broadest classes of tractable global optimisation problems, and has found applications in areas such as polynomial and fractional optimisation \cite{tuy_nonconvex_quadratic, tuy_polynomial_fractional}, as well as resource allocation in communications \cite{monotone_opt_comms_book} and sensor network planning \cite{tuy_sensor_cover}.

\paragraph{Monotonic optimisation}
The systematic study of monotonic optimisation originates with \cite{monotone_opt_tuy, monotone_multi_opt}, which established the canonical problem form and its solution via outer approximation, consolidating an earlier line of work on problems with low-dimensional non-convex structures \cite{low-rank_nonconvex_book}; see \cite{tuy_book, global_opt_book} for a broader treatment of related global optimisation problems.
These methods rest on the observation that a monotone function defined over a rectangular domain attains its optimum at one of the domain's vertices.
This property underpins the \emph{polyblock outer-approximation algorithm} (POA) \cite{monotone_opt_tuy}: a branch-and-bound method which iteratively refines a union of rectangles, or a \emph{polyblock}, containing the feasible set.
Each iteration removes a cone which is provably infeasible, in direct analogy to the separating hyperplanes used by polyhedral outer-approximation methods for concave minimisation \cite{polyhedral-outer-approximation}.
Subsequent work extended the framework in several directions, including POA variants which combine polyblock refinement with additional cuts \cite{monotonic_opt_tuy2} and problems exhibiting mixed monotonicity \cite{partial_montone}.
A recurring difficulty is that POA may fail to terminate when the feasible set is thin or nearly degenerate, motivating the robust reformulation of \cite{sit_algorithm}, which restores termination under continuity assumptions.

Much of the practical development of these methods has been driven by wireless resource allocation, where monotonicity arises naturally in rate and interference expressions \cite{montone_opt_MISO, montone_opt_beamform}.
This literature has produced most of the reported large-scale experience with POA, but the resulting implementations are typically specialised to a particular problem family rather than offered as general-purpose solvers.

\paragraph{Scalability and software}
Despite this activity, POA's practical applicability remains largely restricted to low-dimensional settings.
The principal limitation is that the number of vertices describing the outer approximation can grow exponentially, leading to substantial memory requirements and increasingly expensive subroutines.
This growth is more severe than in conventional branch-and-bound: the number of nodes generated by a single POA refinement scales with the problem's dimension, whereas branch-and-cut methods for mixed-integer programming, for instance, typically generate a fixed number of nodes per expansion, usually just two.
The considerable engineering effort invested in such solvers \cite{baron, scip_suite} has no counterpart in monotonic optimisation, and we are not aware of any publicly available software for solving general monotonic optimisation problems.
Consequently, the implementation techniques which make branch-and-bound practical have received little attention in this setting.

\paragraph{Contributions}
In this paper, we address these scalability limitations through a series of improvements to both the POA algorithm and its implementation.
We first present a generalisation of POA which allows for more freedom in how the algorithm explores its search space, and show that this naturally leads to a balanced monotonicity cut with optimal worst-case performance (Section~\ref{ssec:balanced_anchors}).
We then introduce a relaxed optimality condition which improves POA's robustness, guaranteeing finite termination with an explicit iteration bound (Section~\ref{ssec:relaxed_opt}).
We also propose a vectorisation scheme that refines several vertices concurrently, in the spirit of the node-level parallelism exploited by parallel branch-and-bound solvers \cite{parallel_bnb_survey, parallel_milp_solvers} (Section~\ref{ssec:vectorising}).
To support these developments, we present a tree-based implementation of POA (Section~\ref{sec:implementation}) that substantially accelerates its core subroutines, enabling larger, more expressive approximations within the same compute budget.

To make our contributions accessible, we introduce \texttt{polyblocks} (Section~\ref{sec:software}): a small Python package implementing the algorithms developed in this paper.
Beyond providing these reference implementations, the package exposes a flexible interface that allows users to implement custom POA variants with minimal effort.
Ablation experiments in Section~\ref{sec:results} demonstrate the effectiveness of our proposed improvements over a baseline POA implementation.

\section{Preliminaries}
We begin by fixing notation and recalling the ingredients of monotonic optimisation used throughout the paper: the canonical problem form, polyblocks as outer approximations of normal sets, and the monotone projection by which those approximations are refined.

\paragraph{Notation}
For $x, y \in \mathbb{R}^n$, we denote the partial order $x_i \leq y_i, \ \forall i$ by $x \leq y$ and define $x \geq y,\, x < y,\, \text{and } x > y$ similarly.
We define the closed box $[x, y]:= \Pi_i [x_i, y_i]$, with open and half-open boxes defined analogously.
For a point $x \in \mathbb{R}^n$ and a set $S \subset \mathbb{R}^n$, we define the rectangular sets
$$
P(x) := (-\infty, x], \ \text{ and } \
P(S) := \bigcup\nolimits_{x\in S} P(x).
$$
If $V \subset \mathbb{R}^n$ is a finite set, then $P(V)$ is said to be a \emph{polyblock}.
For any $n \in \mathbb{N}$, let $[n] = \{1,\hdots, n \}$ and let $\unitvector \in \mathbb{N}^n$ denote the vector with unit components.

\paragraph{Monotone functions and normal sets}
A function $f:\mathbb{R}^n \to \mathbb{R}$ is \emph{increasing} if $f(x) \leq f(y)$ whenever $x \leq y$.
When used as constraints, increasing functions induce a corresponding closure property under the partial order.
We call a set $S \subset \mathbb{R}^n$ \emph{normal} when
$$x \in S \implies P(x) \subset S,$$
or equivalently when $P(S) = S$.
The set is \emph{co-normal} when its complement is normal, or equivalently when $x \in S$ implies $[x, \infty) \subset S$.
Both classes arise directly from constraints on monotone functions: if $f$ is increasing, then the sublevel set $\{x\in \mathbb{R}^n: f(x) \leq 0\}$ is normal, while the superlevel set $\{x\in \mathbb{R}^n: f(x) \geq 0\}$ is co-normal.
Polyblocks belong to the same class: each $P(x)$ is normal by construction, and normality is preserved under unions, so $P(V)$ is a normal set specified by the finitely many vertices in $V$.
Polyblocks can therefore act as tractable approximations of a normal set, which is the role they play in Section~\ref{sec:POA_intro}.

\paragraph{Monotonic optimisation problems}
Let $f$ be an increasing function on $\mathbb{R}^n$, $\mathcal{G} \subset \mathbb{R}^n$ be a closed normal set, and $\mathcal{H} \subset \mathbb{R}^n$ be a closed co-normal set.
We consider monotonic optimisation problems in the canonical form:
\begin{align}~\label{eq:monotone-canon}
    \max_{x \in [a,b] } f(x), \ \ \text{s.t. } x \in \mathcal{G} \cap \mathcal{H},
\end{align}
where $a,b \in \mathbb{R}^n$.
If $a \notin \mathcal{G}$, then the normality of $\mathcal{G}$ implies $[a, \infty) \cap \mathcal{G} = \emptyset$ and the problem is infeasible.
We therefore assume throughout that $a\in\mathcal{G}$.
An analogous argument also justifies the assumption $b \in \mathcal{H}$.
We only assume black-box access to the function $f$, as well as feasibility oracles for querying $x \in \mathcal{G}$ and $x \in \mathcal{H}$.

Despite its simple form, \eqref{eq:monotone-canon} is considerably more general than it first appears.
In particular, it is equivalent to the seemingly larger class of difference-of-monotone problems:
\begin{align*}
\begin{split}
    \max_{x \in [a,b] } \quad &f_1(x) - f_2(x) \\
    \ \text{s.t.} \quad & g_j(x) - h_j(x) \leq 0,\ j \in [m],\\
\end{split}
\end{align*}
where all functions $f_j,\, g_j,\, h_j$ are increasing and each variable $x_i$ may be either continuous or discrete \cite{tuy_book}.
This generality is what makes monotonic optimisation broadly applicable, and equally what makes efficient global solution methods difficult to develop.

\paragraph{Monotone projections}
Let $\mathcal{G}\subset\mathbb{R}^n$ be a closed normal set and let $y\in\mathcal{G}$ be a feasible anchor point.
The monotone projection of a point $x > y$ onto $\mathcal{G}$ with respect to $y$ is defined as
\begin{equation} \label{eq:mono_proj}
    \pi_\mathcal{G}(x;y) := r_\mathcal{G}(x; y)\, x + (1 - r_\mathcal{G}(x; y))\,y,
\end{equation}
where
$$r_\mathcal{G}(x; y) := \sup \{r \in [0,1]: r x + (1 - r)y \in \mathcal{G}\}$$
is the largest coefficient for which the convex combination of $x$ and $y$ remains in $\mathcal{G}$.
Monotone projections are flanked by two cones, one lying inside $\mathcal{G}$ and the other outside it.
This property is used both to compute $\pi_\mathcal{G}$ and to discard infeasible points in the outer approximation:
\begin{proposition}\label{prop:cone_separate}
    Let $\mathcal{G} \subset \mathbb{R}^n$ be a closed normal set, let $y \in \mathcal{G}$, and let $x > y$ with $x \notin \interior\,\mathcal{G}$.
    The projection $z := \pi_\mathcal{G}(x; y)$ then satisfies
    $$(-\infty, z] \subset \mathcal{G}, \qquad (z, \infty) \cap \mathcal{G} = \emptyset.$$
\end{proposition}
\begin{proof}
    As $\mathcal{G}$ is closed, the supremum defining $r_\mathcal{G}(x; y)$ is attained, so $z \in \mathcal{G}$ and normality yields the first inclusion, $P(z) \subset \mathcal{G}$.
    For the second claim, suppose such a point $w \in (z, \infty) \cap \mathcal{G}$ exists.
    By normality, $P(w) \subset \mathcal{G}$, and as $z < w$ this would make $z$ an interior point of $\mathcal{G}$.
    That is impossible when $r_\mathcal{G}(x; y) = 1$, since $z$ is then the point $x$, which is not interior by assumption.
    It is also impossible when $r_\mathcal{G}(x; y) < 1$, as an interior $z$ must have a neighbourhood of feasible points which intersects the segment from $z$ to $x$, contradicting the maximality of $r_\mathcal{G}(x; y)$.
\end{proof}

Since the line segment from $y$ to $x$ is strictly increasing in every component, it is contained in the union of these two cones; every point below $z$ on this segment is feasible, while every point above is infeasible.
Hence, projections can be computed using a bisection search on the parameter $r$, which returns the last iterate verified to lie in $\mathcal{G}$ so that the computed projection is always feasible.
The infeasible cone $(z, \infty)$ serves a second purpose: analogous to a hyperplane separation for convex sets, it identifies the region removed from the outer approximation at each iteration of the algorithm we consider.

\section{The Polyblock Outer-Approximation algorithm}\label{sec:POA_intro}
We now present a generalisation of the original POA algorithm \cite{monotonic_opt_tuy2}, which allows more freedom in the choice of projection anchors.
We begin with the general branch-and-bound procedure that POA follows, before detailing the outer-approximation procedure itself.

\subsection{The branch-and-bound procedure}
Let $\mathcal{F} = \mathcal{G} \cap \mathcal{H} \cap [a, b]$ be the feasible region of problem~\eqref{eq:monotone-canon}.
Given some tolerance $\epsilon >0 $, our goal is to find a feasible solution $\hat{z} \in \mathcal{F}$ satisfying
\begin{equation}\label{eq:eps_optimal}
    f(\hat{z}) + \epsilon > f(x), \ \forall x\in \mathcal{F},
\end{equation}
that is, an $\epsilon$-optimal solution.

POA achieves this using a branch-and-bound strategy.
At each iteration $t$ it generates a candidate point $z^t\in\mathcal G$, which may or may not satisfy the remaining constraints $\mathcal H \cap [a, b]$.
The algorithm also maintains the best feasible candidate found so far
\begin{equation} \label{eq:best_candidate}
\hat z^t \in \argmax_{z^\tau \in \mathcal{F},\; \tau < t} f(z^\tau),
\end{equation}
whenever such a point exists, together with its objective value
$$
\hat f^t :=
\begin{cases}
f(\hat z^t), &
\{z^\tau: z^\tau \in \mathcal{F},\, \tau<t \} \neq \emptyset,\\
-\infty, & \text{otherwise}.
\end{cases}
$$
When $\hat{z}^t$ exists, only feasible points with objective values $ \geq \hat f^t+\epsilon$ violate its $\epsilon$-optimality \eqref{eq:eps_optimal}.
Consequently, at each iteration $t$, it is sufficient to continue searching for candidates only within the reduced feasible region
$$
\mathcal F^t :=
\{ x \in \mathcal{F} : f(x) \geq \hat f^t+\epsilon \}.
$$
If $\mathcal F^t=\emptyset$, then either the original problem is infeasible (when no feasible candidate has been found) or $\hat z^t$ is $\epsilon$-optimal.

To represent the remaining search space $\mathcal{F}^t$, POA constructs a sequence of polyblock outer-approximations:
$$P(V^t) \supset \mathcal{F}^t,$$
where each $V^t \subset [a, b]$ is a finite set of vertices.
Since $f$ is increasing, its maximum over the polyblock $P(V^t)$ is attained at one of its vertices:
\begin{align}\label{eq:maximiser}
\hat{x}^t \in \argmax_{x \in V^t} f(x)
\subset
\argmax_{x \in P(V^t)} f(x).
\end{align}
Thus, $f(\hat{x}^t)$ provides an upper bound on all remaining feasible solutions, while $\hat{f}^t$ provides a corresponding lower bound:
$$f(\hat{x}^t) \geq \sup_{x \in \mathcal{F}^t} f(x) > \hat{f}^t.$$

The POA algorithm repeatedly refines the polyblock $P(V^t)$ to tighten these bounds while simultaneously improving the incumbent solution $\hat z^t$.
The algorithm terminates once $V^t=\emptyset$, at which point $\mathcal F^t=\emptyset$ and the incumbent solution, if one exists, is $\epsilon$-optimal.

\subsection{Polyblock refinement process}\label{ssec:poly_refinement}
POA is initialised with a single polyblock vertex $V^0=\{b\}$.
Each iteration replaces the current vertex set $V^t$ with a refined vertex set $V^{t+1}$ that represents a tighter outer approximation of the remaining search space $\mathcal{F}^{t+1}$.

The refinement begins by selecting the maximal vertex
$\hat{x}^t\in\arg\max_{x\in V^t}f(x),$
as defined in \eqref{eq:maximiser}, together with any projection anchor $y^t\in\mathcal{G}$ lying strictly below it, $y^t < \hat{x}^t$, and satisfying
\begin{equation}
\label{eq:anchor_constr}
    \frac{\min_i(\hat{x}^t_i-y^t_i)}
         {\max_i(\hat{x}^t_i-y^t_i)}
    \geq\rho,
\end{equation}
where $\rho\in(0,1]$ is fixed throughout the algorithm.
Since $a\in\mathcal{G}\cap P(\hat{x}^t)$, a simple choice is
\begin{equation}
\label{eq:old_anchor}
    y^t = a - \frac{\rho\|b-a\|_\infty}{1-\rho}\,\unitvector, \ \ \text{for any $\rho\in(0,1)$.}
\end{equation}
The anchor $y^t$ is then used to obtain the candidate $z^t$ by projecting $\hat{x}^t$ onto $\mathcal{G}$:
$$z^t := \pi_{\mathcal{G}}(\hat{x}^t; y^t).$$
As we show below, each polyblock vertex set satisfies $V^t \cap \interior\, \mathcal{G} = \emptyset$, so Proposition~\ref{prop:cone_separate} applies to $\hat{x}^t$ and the cone $(z^t,\infty)$ is disjoint from $\mathcal{G}$, and hence from the feasible region $\mathcal{F}$.
As a result, this cone can be removed from the current polyblock $P({V^{t}})$ without excluding any feasible points:
$$P(V^t) \backslash (z^t,\infty) \supset \mathcal{F}^t.$$

The cone $(z^t,\infty)$ partitions the current vertex set into
$$
F^t = V^t \backslash (z^t,\infty),
\quad
I^t = V^t\cap(z^t,\infty),
$$
corresponding to feasible and infeasible vertices, respectively.
The refined vertex set is obtained by leaving $F^t$ unchanged while replacing all vertices in $I^t$ to remove the infeasible portion of their associated polyblocks.

Each infeasible vertex $s\in I^t$ is replaced by $n$ new vertices, each obtained by reducing one coordinate of $s$ to the corresponding coordinate of $z^t$.
For $u,v \in \mathbb{R}^n$ and $i \in [n]$, write
$$ \varrho_i(u,v) := (u_1,\hdots,u_{i - 1}, v_i, u_{i+1}, \hdots, u_n)$$
for the vector obtained from $u$ by replacing its $i$-th component with that of $v$; the new vertices are then given by $\{\varrho_i(s, z^t)\}_{i \in [n]}$.
These vertices satisfy
$$ P(\{\varrho_i(s, z^t)\}_{i\in[n]}) = P(s) \backslash (z^t,\infty), $$
and therefore exactly describe the polyblock obtained by removing the infeasible cone from $P(s)$.

Taking the union over all of $I^t$ gives a valid refinement representation, but a wasteful one: a new vertex dominated by another, $\varrho_i(s, z^t) \leq u$, contributes nothing, as $P(u) \supseteq P(\varrho_i(s, z^t))$, and may be discarded.
Such vertices can be identified without building the union first, using the characterisation from~\cite{monotone_opt_tuy}:
\begin{equation} \label{eq:refined_vertices}
    W^t := \{ \varrho_i(s, z^t) : \varrho_i(s, z^t) \notin P(\bar{I}^t\backslash\{s\}),\, s\in I^t, \, i\in[n] \},
\end{equation}
where
$$\bar{I}^t := V^t \cap [z^t, \infty).$$
This test compares each new vertex against the old vertex set alone, rather than against all other new vertices.
This suffices as domination is inherited by the refinement: if $u \in \bar{I}^t \backslash \{s\}$ dominates $\varrho_i(s, z^t)$, then so does $\varrho_i(u, z^t)$, as the two agree in their $i$-th component.
The vertices which survive reproduce the refined outer approximation exactly:
$$P(F^t \cup W^t) = P(V^t) \backslash (z^t, \infty).$$

Two further prunings are available at no cost to correctness.
Since $\mathcal H^c$ is normal, any vertex outside $\mathcal H\cap[a,b]$ may be removed without excluding feasible points, leaving the feasible new vertices
$$N^t := W^t \cap \mathcal{H} \cap [a, b].$$
Likewise, any vertex satisfying $f(s)<\hat f^{t + 1}+\epsilon$ cannot belong to $\mathcal F^{t+1}$.
Applying both pruning steps gives the vertex update
\begin{equation} \label{eq:vertex_update}
    V^{t+1} = \{ s \in F^t \cup N^t: f(s) \geq \hat{f}^{t + 1} + \epsilon \}
\end{equation}
which retains the guarantee $V^{t+1} \supset \mathcal{F}^{t+1}$.
It also closes the induction assumed above: each new vertex satisfies $\varrho_i(s, z^t) \in [z^t,\infty)$ and so avoids $\interior\mathcal G$, while those carried over from $F^t$ avoid it by hypothesis, preserving $V^{t+1} \cap \interior \mathcal{G} = \emptyset.$
A similar induction gives $V^t \subset \mathcal{H}\cap[a,b]$, and hence $\hat{x}^t \in \mathcal{H}\cap[a,b]$, as $b$ lies in $\mathcal{H}\cap[a,b]$ by assumption and the pruning to $N^t$ imposes the same constraint on each new vertex.

This refinement procedure is repeated until $V^t=\emptyset$.
Algorithm~\ref{alg:poa} summarises the complete POA algorithm, while Figure~\ref{fig:tree_polyblock}\subref{fig:vertex_form} illustrates the construction of the first three polyblock approximations.

\begin{figure}[h]
    \centering
    \begin{minipage}[b]{0.50\textwidth}
        \centering
        \includegraphics[width=\linewidth, trim=26 33 293 119, clip]{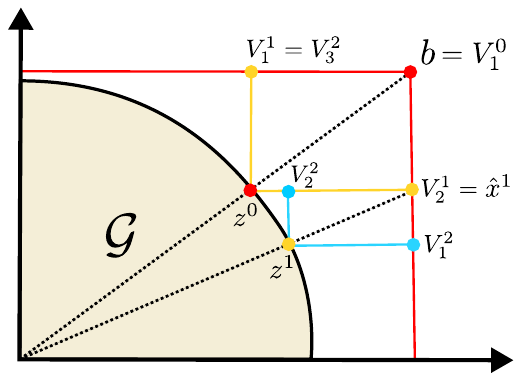}
        \subcaption{Polyblock refinement.}
        \label{fig:vertex_form}
    \end{minipage}
    \hspace{0.04\textwidth}
    \begin{minipage}[b]{0.285\textwidth}
        \centering
        \includegraphics[width=\linewidth, trim=382 41 69 144, clip]{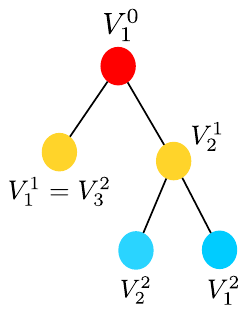}
        \subcaption{Vertex tree.}
        \label{fig:tree_form}
    \end{minipage}
    \caption{
    The first two POA iterations on a normal set $\mathcal{G}\subset\mathbb{R}^2$, where $V^t_i$ denotes the $i$-th vertex of the polyblock at iteration $t$ and colours indicate the iteration at which each vertex was generated.
    The initial vertex $V^0_1 = b$ is projected onto $\mathcal{G}$ to give $z^0$ and replaced by $V^1_1$ and $V^1_2$; projecting the new maximal vertex $\hat{x}^1 = V^1_2$ gives $z^1$ and a second refinement.
    The resulting vertex tree is discussed in Section~\ref{ssec:vertex_tree}.
    }
    \label{fig:tree_polyblock}
\end{figure}

\begin{algorithm}
    \caption{The Polyblock Outer-Approximation algorithm}
    \label{alg:poa}
    \small\begin{algorithmic}[0]
    \State \textbf{Inputs:} Problem data: $(f,\, \mathcal{G},\, \mathcal{H},\, a,\, b)$,
    Solver parameters: $(\epsilon,\, \rho)$
    \State \textbf{Output:} An $\epsilon$-optimal solution,
    or NULL if $\mathcal{F}$ is empty

    \State $V \gets \{b\}$
    \State $\hat{z} \gets \text{NULL}$
    \State $\hat{f} \gets -\infty$
    \While{$V \neq \emptyset$}
        \State $\hat{x} \gets \argmax_{x\in V} f(x)$
        \State $y \gets \text{any}\ y\in \mathcal{G}\ \text{with}\ y < \hat{x}\ \text{satisfying \eqref{eq:anchor_constr}}$
        \State $z \gets \pi_\mathcal{G}(\hat{x}; y)$
        \If{$f(z) > \hat{f}$ and $z \in \mathcal{H} \cap [a,b]$}
            \State $\hat{z} \gets z$
            \State $\hat{f} \gets f(z)$
        \EndIf
        \State $V \gets$ Update as \eqref{eq:vertex_update}
    \EndWhile
    \Return $\hat{z}$
\end{algorithmic}
\end{algorithm}

If a feasible incumbent exists when the algorithm terminates, it is $\epsilon$-optimal.
Otherwise, we have $\mathcal{F} = \mathcal{F}^t = \emptyset$ and the original problem is infeasible.
When the algorithm does not terminate finitely, \cite{monotonic_opt_tuy2} showed that, under upper-continuity of $f$ and assuming \eqref{eq:old_anchor}, the sequences $(\hat{x}^t)_t$ and $(z^t)_t$ both converge to optimal solutions.
We provide a slight generalisation of this result:
\begin{theorem} \label{thm:POA_old}
    For a set $S \subset \mathbb{R}^n$ and a point $x$, write $d(x, S) := \inf_{s \in S} \|x - s\|_\infty$.
    Assume that $f$ is upper-continuous on $\mathcal{H}$ and let $\mathcal{F}^\star := \argmax_{x \in \mathcal{F}} f(x)$ be the solution set of \eqref{eq:monotone-canon}.
    If Algorithm \ref{alg:poa} is infinite, then $\mathcal{F} \neq \emptyset$, $d(\hat{x}^t, \mathcal{F}^\star) \to 0$, and $d(z^t, \mathcal{F}^\star) \to 0$.
\end{theorem}
\begin{proof}

    For all $i \in [n]$ and $t \geq 0$ we compute:
    \begin{align*}
        \hat{x}^t_i - z^t_i &= \hat{x}^t_i - \left[r_{\mathcal{G}}(\hat{x}^t; y^t)\,\hat{x}^t_i  + (1-r_{\mathcal{G}}(\hat{x}^t; y^t))\, y^t_i \right] \\
        & = (1 - r_{\mathcal{G}}(\hat{x}^t; y^t))\, (\hat{x}^t_i - y^t_i).
    \end{align*}
    From definition \eqref{eq:mono_proj}, it follows that
    $$1 - r_{\mathcal{G}}(\hat{x}^t; y^t) = \frac{\|\hat{x}^t - z^t\|_\infty}{\|\hat{x}^t - y^t \|_\infty},$$
    which bounds the componentwise gap between the vertex and its projection:
    \begin{align*}
    \begin{split}
        \min_i (\hat{x}^t_i - z^t_i) &= \frac{\min_i(\hat{x}^t_i - y^t_i)}{\|\hat{x}^t - y^t\|_\infty}\, \|\hat{x}^t - z^t\|_\infty \\
        & \geq \rho\, \|\hat{x}^t - z^t\|_\infty.
    \end{split}
    \end{align*}
    We use this to show that $\hat{x}^t - z^t \to 0$.
    Indeed, if $\| \hat{x}^t - z^t\|_\infty \geq \eta$ holds for some $\eta > 0$ and infinitely many $t$, then the above bound implies $\min_i (\hat{x}^t_i - z^t_i) \geq \rho\,\eta > 0$ also occurs infinitely often.
    Since $z^t \leq \hat{x}^t - \rho\,\eta\, \unitvector$ and POA ensures $\hat{x}^s \notin (z^t, \infty)$ for $s > t$, this implies the rectangles $(\hat{x}^t - \rho\, \eta\, \unitvector,  \hat{x}^t) \subset (z^t,  \infty)$ are pairwise disjoint.
    This contradicts the boundedness of $(\hat{x}^t)_t \subset [a, b] \cap \mathcal{H}$, so we must have $\hat{x}^t - z^t \to 0$.

    The boundedness of $(\hat{x}^t)_t$ also implies that $d(\hat{x}^t, C) \to 0$ where $C \subset \mathcal{H}\cap[a,b]$ is the set of all limit points of $(\hat{x}^t)_t$.
    By $\hat{x}^t - z^t \to 0$, we also have $d(z^t, C) \to 0$, and $C \subset \mathcal{F}$ as $(z^t)_t \subset \mathcal{G}$.
    $C$ is non-empty as bounded sequences must have at least one limit point, which implies $\mathcal{F}$ is also non-empty.
    It remains to show that $C \subset \mathcal{F}^\star$, which follows from the continuity assumption along with the fact that $f(\hat{x}^t) \geq f(x), \ \forall t > 0,\, x\in \mathcal{F}$, hence $f(c) \geq f(x), \ \forall c \in C,\, x\in \mathcal{F}$.
    \end{proof}

\subsection{The vertex tree induced by refinement}\label{ssec:vertex_tree}
The refinement procedure described above naturally induces a tree structure over the set of all polyblock vertices generated up to the current iteration.
This view carries the analysis as well as the implementation: it recasts the progress made at each iteration as growth of a tree, which later yields explicit complexity bounds, and it supplies the data structure underlying the solver proposed in Section~\ref{sec:implementation}.

As with other branch-and-bound algorithms, POA can be interpreted as implicitly exploring a tree whose nodes represent subsets of the search space.
This tree is constructed in a top-down manner by iteratively expanding its leaf nodes.
In POA, each tree node represents a polyblock vertex, with the root corresponding to the initial vertex $b$.
At iteration $t$, the active leaf nodes are precisely the current vertex set $V^t$, where each leaf $s\in V^t$ represents the rectangular region $P(s)$.

Exploring a leaf $s\in I^t$ replaces it with its children
$$ \xi^t(s) := \{\varrho_i(s, z^t) : \varrho_i(s, z^t)\in V^{t+1},\; i\in[n]\}, $$
a set of up to $n$ new vertices, each corresponding to one of the refined polyblocks introduced in the previous subsection.
The union of the regions represented by these children satisfies
$$ P(\xi^t(s)) \subset P(s) \backslash (z^t, \infty), $$
and is therefore a strict subset of the region represented by the parent node.
Consequently, every explored leaf corresponds to a refinement of the remaining search space.
Since the maximal vertex $\hat{x}^t$ is explored at every iteration, the overall search space represented by the active leaves decreases strictly as the algorithm progresses.
The algorithm terminates once no active leaf nodes remain, that is, when $V^t=\emptyset$.

Formally, let
$$ T^t := \bigcup\textstyle_{i=0}^{t}V^i $$
denote the set of all vertices generated up to iteration $t$.
The parent-child relationships are then given by
$$ E^t := \{(s, u) : u \in \xi^\tau(s),\; s \in V^\tau,\; \tau \in [t]\}, $$
and the resulting vertex tree is
$$ \mathcal T^t:=(T^t,E^t).$$

This correspondence is illustrated in Figure~\ref{fig:tree_polyblock}\subref{fig:tree_form}, where each refined vertex becomes the parent of the vertices replacing it, while the unrefined vertex $V^1_1$ remains a leaf and persists into the next iteration as $V^2_3$.

\section{Improved methods}\label{sec:improvements}
In this section, we propose three improvements to the basic POA procedure of the previous section: balanced anchors, relaxed optimality conditions, and vectorisation.
We treat each in turn, discussing its advantages, before combining all three in Algorithm~\ref{alg:poa_v3}.

\subsection{Balanced anchors}\label{ssec:balanced_anchors}

POA requires a sequence of anchor points satisfying \eqref{eq:anchor_constr} to compute the projections $z^t$.
Previous implementations typically used a fixed, iteration-independent anchor, such as \eqref{eq:old_anchor}.
The choice is not innocuous, as the convergence guarantee of Theorem~\ref{thm:POA_old} depends explicitly on $\rho$, which we now show to bound the volume removed from the search space at each iteration.

The role of $\rho$ becomes clear once the refinement is expressed in terms of the anchor.
Projections shift $\hat{x}^t$ along the line segment towards $y^t$, so for any admissible anchor,
$$ \hat{x}^t_i - z^t_i = (1 - r_{\mathcal{G}}(\hat{x}^t; y^t))\, (\hat{x}^t_i - y^t_i), \quad \forall i\in[n]. $$
The region $(z^t,\hat{x}^t]$ removed at iteration $t$ is therefore a scaled copy of the box spanned by $\hat{x}^t$ and its anchor, and inherits its aspect ratio:
$$ \frac{\min_i(\hat{x}^t_i - z^t_i)}{\|\hat{x}^t - z^t\|_\infty} = \frac{\min_i(\hat{x}^t_i - y^t_i)}{\max_i(\hat{x}^t_i - y^t_i)} \geq \rho. $$
For a fixed projection distance $M := \|\hat{x}^t - z^t\|_\infty$, the removed volume is therefore at least $\rho^{n-1}M^n$.
Small values of $\rho$ permit increasingly elongated hyperrectangles, which remove smaller portions of the search space.

This guarantee is consequently maximised by an anchor lying at equal distance from $\hat{x}^t$ in every coordinate, which motivates the adaptive anchor
\begin{equation} \label{eq:new_anchors}
    y^t = \hat{x}^t - \|\hat{x}^t - a\|_\infty\,\unitvector,
\end{equation}
which lies strictly below $\hat{x}^t$, except for the degenerate case where $\hat{x}^t = a$, and is feasible since $y^t\leq a\in\mathcal G$.
Every component of $\hat{x}^t - y^t$ is identical, so every coordinate is reduced by the same amount during the projection and $(z^t,\hat{x}^t]$ is a cube.
The resulting value $\rho = 1$ is the maximum possible value of the ratio \eqref{eq:anchor_constr}.
This is also where the improvement enters the analysis, as the proof of Theorem~\ref{thm:POA_old} invokes $\rho$ only through the ratio bound above, which \eqref{eq:new_anchors} attains with equality.

The proposed anchors are also uniformly bounded:
$$
|y^t_i| \leq
|\hat{x}^t_i| + \|\hat{x}^t - a\|_\infty \leq
|b_i| +\|b - a\|_\infty,
\ \forall i\in [n].
$$
Consequently, the line segment joining $y^t$ and $\hat{x}^t$ has uniformly bounded length, ensuring that the complexity of computing $\pi_{\mathcal{G}}(\hat{x}^t; y^t)$ by bisection remains bounded throughout the algorithm.

\subsection{Relaxed optimality conditions}\label{ssec:relaxed_opt}

Balanced anchors improve the rate at which the outer approximation tightens, but they leave untouched a more basic difficulty: whether POA terminates at all.
POA often struggles when the co-normal constraint set $\mathcal{H}$ is small or contains isolated regions near the boundary of $\mathcal{G}$.
In particular, when $[a, b] \cap \mathcal{H}\not\subseteq \mathcal{G}$, the algorithm may repeatedly generate infeasible projections even when the original problem is feasible, and hence fail to terminate.
Furthermore, if the optimal solution lies in a small isolated region of $\mathcal{F}$, or directly on the boundary of $\mathcal{G}$, the convergence behaviour can become highly sensitive to numerical tolerances \cite{monotonic_opt_tuy2}.
We now propose a relaxed optimality condition which addresses these issues.

\subsubsection{The \texorpdfstring{$\delta$}{delta}-relaxed problem}
To improve robustness, we introduce a relaxed optimality criterion which disregards solutions lying within $\delta$ of the boundary of $\mathcal{G}$.
For $\delta\geq0$, define the $L_\infty$-erosion of $\mathcal{G}$ as
$$\mathcal{G}_\delta := \{x\in\mathcal{G}:B_\delta(x) \subset \mathcal{G}\},$$
where
$$B_\delta(x):=\{x+v:\|v\|_\infty\leq\delta\}$$
is the closed $\delta$-ball centred at $x$.
We say that $x^\star$ is a \emph{$\delta$-relaxed} solution of \eqref{eq:monotone-canon} if
\begin{equation}\label{eq:monotone_relax}
x^\star \in \mathcal{F},\
f(x^\star)\geq f(x),\quad \forall x\in\mathcal{F}_\delta,
\end{equation}
where
$$ \mathcal{F}_\delta := \mathcal{G}_\delta \cap \mathcal{H} \cap [a,b]. $$
In other words, the relaxed solution is a feasible point that is optimal over only the subset of points which remain feasible under any $\delta$-perturbation.
Allowing the tolerance of \eqref{eq:eps_optimal} in addition, we call $x^\star$ a \emph{$\delta$-relaxed $\epsilon$-optimal} solution when $x^\star\in\mathcal{F}$ and $f(x^\star) + \epsilon > f(x)$ for all $x\in\mathcal{F}_\delta$.
We take the relaxed problem to be infeasible if $\mathcal{F}_\delta$ is empty.

Since $\mathcal{G}$ is normal, the inclusion criterion for $\mathcal{G}_\delta$ reduces to a test on a single point.
The box $B_\delta(s) = [s - \delta\,\unitvector,\, s + \delta\,\unitvector]$ has greatest element $s + \delta\,\unitvector$, so $B_\delta(s) \subset P(s + \delta\,\unitvector)$ and
\begin{align*}
    s \in \mathcal{G}_\delta &\iff B_\delta(s) \subset \mathcal{G} \\
    &\iff s + \delta\,\unitvector \in \mathcal{G},
\end{align*}
where the reverse implication holds because normality places the whole of $P(s + \delta\,\unitvector)$ inside $\mathcal{G}$.
Moreover, $\mathcal{G}_\delta$ is itself normal.
Indeed,
\begin{align*}
    s \in \mathcal{G}_\delta &\iff s + \delta\,\unitvector \in \mathcal{G}, \\
    & \iff v + \delta\,\unitvector \in \mathcal{G},\  \forall v \leq s \\
    & \iff v \in \mathcal{G}_\delta, \ \forall v \leq s.
\end{align*}

The relaxed problem retains the monotonic structure required by POA while avoiding the pathological behaviour caused by solutions lying arbitrarily close to the boundary of $\mathcal{G}$.

\subsubsection{Modifications to POA}
Since every optimal solution of \eqref{eq:monotone-canon} is also a $\delta$-relaxed solution, POA can be applied directly to the relaxed problem.
However, the relaxation also enables three modifications which substantially improve the efficiency of the algorithm.

\paragraph{Tighter polyblock approximation}
Since optimisation is restricted to $\mathcal{F}_\delta$, it is sufficient to use the tightened outer approximation
$$P(V^t)\supset\mathcal{F}_\delta^t,$$
where
$\mathcal{F}^t_\delta := \{x\in\mathcal{F}_\delta: f(x) \geq \hat{f}^t + \epsilon\}$
is the relaxed counterpart of the reduced feasible region $\mathcal{F}^t$ of Section~\ref{sec:POA_intro}.
This approximation is constructed following the refinement procedure of Section~\ref{ssec:poly_refinement}, replacing $\mathcal{G}$ with its eroded counterpart $\mathcal{G}_\delta$.
Projections onto $\mathcal{G}_\delta$ require no new machinery: the feasibility oracle follows from the simplified inclusion criterion $x \in \mathcal{G}_\delta \iff x + \delta\, \unitvector \in \mathcal{G}$, and a valid $\mathcal{G}_\delta$ anchor is obtained by shifting any $\mathcal{G}$ anchor backwards by $\delta$.
In particular, the balanced anchor \eqref{eq:new_anchors} of a vertex $\hat{x}^t \in V^t$ becomes
$$y^t = \hat{x}^t - (\|\hat{x}^t - a\|_\infty + \delta)\, \unitvector.$$
From here onwards, the notation of Section~\ref{ssec:poly_refinement} refers throughout to the tightened approximation; in particular the projection is
$$z^t = \pi_{\mathcal{G}_\delta}(\hat{x}^t; y^t),$$
and $V^t$, $W^t$, $N^t$, $F^t$, $I^t$ denote the corresponding vertex sets.

\paragraph{Partial refinement}
Maintaining the tighter outer approximation suggests refining every vertex contained in the infeasible cone $(z^t, \infty)$.
However, if a vertex $s$ lies close to the boundary of this cone, then the region $(z^t, s]$ removed from the remaining search space is often very small.
Refining such a vertex nevertheless generates up to $n$ new vertices, increasing the size of the polyblock representation.
Since the efficiency of POA depends critically on the number of stored vertices, these refinements provide little benefit relative to their computational cost.
We therefore refine only the vertices lying at least $\delta$ above the projection,
$$I^t_\delta := V^t \cap (z^t + \delta\,\unitvector, \infty),$$
thereby avoiding refinements which remove a region smaller than the relaxation scale $\delta$.
The remaining vertices $F^t_\delta := V^t \backslash I^t_\delta$ are carried over unchanged, taking the place of the feasible vertices $F^t$ of Section~\ref{ssec:poly_refinement}.

After removing redundant vertices, the refined vertex set becomes
\begin{equation}\label{eq:refined_vertices2}
    W^t = \{\varrho_i(s, z^t):
    \varrho_i(s, z^t) \notin P(\bar{I}^t \backslash \{s\} ),\,
    s \in I^t_\delta ,\,
    i\in [n] \},
\end{equation}
which, together with $F^t_\delta$ in place of $F^t$, replaces \eqref{eq:refined_vertices} in the vertex update \eqref{eq:vertex_update}.

Although this modification is not required for correctness, it prevents the algorithm from generating large numbers of vertices in exchange for only negligible reductions in the remaining search space.
As we show below, it also preserves the convergence guarantees of the relaxed algorithm.

\paragraph{Shifted candidates}

The relaxed optimality condition \eqref{eq:monotone_relax} requires only that the returned solution lie in $\mathcal{F}$, not in the smaller set $\mathcal{F}_\delta$.
Candidates therefore need not belong to the eroded set $\mathcal{G}_\delta$, and each may be raised by up to $\delta$ when updating the incumbent, improving its objective value at no additional cost.
In place of the projection $z^t \in \mathcal{G}_\delta$, we thus consider the shifted candidate
$$c^t := \min(z^t + \delta\,\unitvector,\, b),$$
where the minimum is taken componentwise.

This shift preserves feasibility while never decreasing the objective.
Since $c^t \leq z^t + \delta\,\unitvector \in \mathcal{G}$, the normality of $\mathcal{G}$ gives $c^t \in \mathcal{G}$.
Suppose further that the projection is feasible, $z^t \in \mathcal{F}_\delta$.
Then $a \leq z^t \leq c^t \leq b$ gives $c^t \in [a,b]$, while $c^t \geq z^t \in \mathcal{H}$ gives $c^t \in \mathcal{H}$ because $\mathcal{H}$ is co-normal.
Hence $c^t \in \mathcal{F}$, and $f(c^t) \geq f(z^t)$ as $f$ is increasing.
Accordingly, we replace the best candidate \eqref{eq:best_candidate} with
$$\hat z^t = \argmax_{c^\tau \in \mathcal{F},\; \tau < t} f(c^\tau).$$

Algorithm~\ref{alg:poa_v2} summarises the resulting POA algorithm, combining the three $\delta$-relaxation modifications above with the balanced anchor strategy.

\begin{algorithm}
    \caption{The $\delta$-relaxed POA algorithm with balanced anchors}
    \label{alg:poa_v2}
    \small\begin{algorithmic}[0]
    \State \textbf{Inputs:} Problem data: $(f,\, \mathcal{G},\, \mathcal{H},\, a,\, b)$, \ Solver parameters: $(\epsilon,\, \delta)$
    \State \textbf{Output:} A $\delta$-relaxed solution, or NULL if $\mathcal{F}_\delta$ is empty

    \State $V \gets \{b\}$
    \State $\hat{z} \gets$ NULL
    \State $\hat{f} \gets -\infty$
    \While{$V \neq \emptyset$}
        \State $\hat{x} \gets \argmax_{x\in V} f(x)$
        \State $y \gets \hat{x} - (\|\hat{x} - a\|_\infty + \delta)\,\unitvector$
        \State $z \gets \pi_{\mathcal{G}_\delta}(\hat{x}; y)$
        \State $c \gets \min(z + \delta\,\unitvector,\, b)$
        \If{$f(c) > \hat{f}$ and $c \in \mathcal{H} \cap [a,b]$}
            \State $\hat{z} \gets c$
            \State $\hat{f} \gets f(c)$
        \EndIf
        \State $V \gets$ Update as \eqref{eq:vertex_update} with $W^t$ given by \eqref{eq:refined_vertices2}
    \EndWhile
    \Return $\hat{z}$
\end{algorithmic}
\end{algorithm}

\subsubsection{Convergence and complexity}\label{sssec:convergence}
Unlike the standard POA algorithm, the relaxed algorithm with $\delta > 0$ always terminates.
Furthermore, unlike existing robust monotonic optimisation methods \cite{sit_algorithm}, this guarantee does not rely on continuity assumptions and is accompanied by an explicit iteration bound.
We first state the properties of the relaxed algorithm on which the analysis rests: that it is correct whenever it terminates, and that it refines its maximal vertex at every iteration but the last.
\begin{lemma}~\label{lem:poa_properties}
    Let $t \in \mathbb{N}$ be an iteration of Algorithm~\ref{alg:poa_v2}.
    \begin{enumerate}[label=(\roman*)]
        \item If the algorithm terminates at iteration $t$, it returns a $\delta$-relaxed $\epsilon$-optimal solution, or NULL, in which case $\mathcal{F}_\delta = \emptyset$.
        \item Otherwise the maximal vertex lies $\delta$ above its projection, $\hat{x}^t \in I^t_\delta$, and so is refined.
    \end{enumerate}
\end{lemma}
\begin{proof}
    (i) If the returned solution $\hat{z}$ is NULL, then no feasible candidate was found and $\hat{f}^t = -\infty$, so
    $$\mathcal{F}_\delta = \mathcal{F}^t_\delta \subset P(V^t) = \emptyset,$$
    and the relaxed problem is infeasible.
    Otherwise $\hat{z} \in \mathcal{F}$, and
    \begin{align*}
        V^t = \emptyset &\implies \mathcal{F}^t_\delta = \emptyset \\
        &\implies f(\hat{z}^t) + \epsilon > f(x), \ \forall x\in\mathcal{F}_\delta,
    \end{align*}
    hence $\hat{z}$ is a $\delta$-relaxed $\epsilon$-optimal solution.

    (ii) We show that the algorithm terminates at any iteration whose maximal vertex is feasible, so that $\hat{x}^t \notin \mathcal{G}$ otherwise.
    In that case $\hat{x}^t - \delta\, \unitvector \in \mathcal{G}_\delta$, and since this point lies on the line joining $y^t$ and $\hat{x}^t$, the projection cannot fall below it:
    $$z^t \geq \hat{x}^t - \delta\,\unitvector.$$
    The candidate shift therefore recovers at least the vertex itself,
    $$c^t = \min(z^t + \delta\,\unitvector,\, b) \geq \min(\hat{x}^t,\, b) = \hat{x}^t,$$
    using $\hat{x}^t \in V^t \subset [a,b]$.
    This candidate is feasible: $c^t \leq z^t + \delta\, \unitvector \in \mathcal{G}$ gives $c^t \in \mathcal{G}$, $c^t \in [a,b]$ by construction, and $c^t \geq \hat{x}^t \in \mathcal{H}$ gives $c^t \in \mathcal{H}$ by co-normality, as every vertex lies in $\mathcal{H} \cap [a,b]$.
    This candidate is therefore no better than the incumbent solution, $\hat{f}^{t+1} \geq f(c^t)$.
    Since $\hat{x}^t$ maximises $f$ over $P(V^t)$, we have
    $$\hat{f}^{t+1} \geq f(c^t) \geq f(\hat{x}^t) \geq f(s), \quad \forall s \in P(V^t) \supset F^t_\delta \cup N^t,$$
    so every remaining vertex is discarded by the objective pruning in \eqref{eq:vertex_update},
    $$
    V^{t+1} =
    \{s \in F^t_\delta \cup N^t: f(s) \geq \hat{f}^{t+1} + \epsilon\} =
    \emptyset,
    $$
    and the algorithm terminates.

    At any other iteration the maximal vertex is therefore infeasible, so that $\hat{x}^t - \delta\, \unitvector \notin \mathcal{G}_\delta$.
    Its anchor is feasible, as $y^t + \delta\, \unitvector = \hat{x}^t - \|\hat{x}^t - a\|_\infty\, \unitvector \leq a$ gives $y^t \in \mathcal{G}_\delta$ by normality, while the projection is obtained by reducing the components of $\hat{x}^t$ equally,
    $$ \hat{x}^t - z^t = (1 - r_{\mathcal{G}_\delta}(\hat{x}^t; y^t))\, (\|\hat{x}^t - a\|_\infty + \delta)\, \unitvector. $$
    Every component is therefore reduced by the same amount
    $\lambda := (1 - r_{\mathcal{G}_\delta}(\hat{x}^t; y^t))\, (\|\hat{x}^t - a\|_\infty + \delta)$.
    The infeasible point $\hat{x}^t - \delta\, \unitvector$ cannot lie below $z^t = \hat{x}^t - \lambda\, \unitvector$, as Proposition~\ref{prop:cone_separate} would then place it in $\mathcal{G}_\delta$.
    Hence $\hat{x}^t - z^t = \lambda\, \unitvector > \delta\, \unitvector$, that is $\hat{x}^t \in I^t_\delta$.
\end{proof}

We next bound the coordinate sum lost along a path of the vertex tree $\mathcal{T}^t$ introduced in Section~\ref{ssec:vertex_tree}, and deduce from it a bound on the depth of the tree:
\begin{lemma}~\label{lem:vertex_paths}
    Let $\mathcal{T}^t$ be the vertex tree produced by any run of Algorithm~\ref{alg:poa_v2} up to some iteration $t \in \mathbb{N}$,
    and let $s^1, \hdots, s^k \in T^t$ be a path in this tree, so that $(s^l, s^{l+1}) \in E^t$ for all $l \in [k-1]$.
    Then
    \begin{enumerate}[label=(\roman*)]
        \item the path loses at least $\delta$ of coordinate sum per edge,
        $\sum\nolimits_i (s^1_i - s^k_i) \geq (k-1)\, \delta$;
        \item if $s^1$ is the root of the tree, the path is no longer than $D := \lfloor \delta^{-1}\, \|b-a\|_1 \rfloor + 1$ nodes, $k \leq D$.
    \end{enumerate}
\end{lemma}
\begin{proof}
    (i) The left-hand side telescopes, so it is sufficient to prove the case $k=2$.
    A single edge $(s^1, s^2) \in E^t$ arises from refining $s^1$ at some iteration $\tau \in [t]$, so that $s^2 = \varrho_j(s^1, z^\tau)$ for some $j \in [n]$, and the two vertices differ only in their $j$-th component:
    \begin{align*}
        \sum\nolimits_i\, (s^1_i - s^2_i) = s^1_j - z^\tau_j \geq \delta,
    \end{align*}
    where the inequality holds because \eqref{eq:refined_vertices2} refines only vertices satisfying $s^1 > z^\tau + \delta\, \unitvector$.

    (ii) This follows by taking $s^1 = b$, the initial vertex held at the root of the tree, and using $s^k \in [a,b]$:
    $$ (k-1)\, \delta \leq \sum\nolimits_i (b_i - s^k_i) \leq \|b - a\|_1. $$
    Rearranging gives $k \leq \delta^{-1}\, \|b-a\|_1 + 1$, and $k \leq D$ as $k$ is an integer.
\end{proof}
We now state the main result:
\begin{theorem} \label{thm:POA_new}
    Algorithm~\ref{alg:poa_v2} terminates in no more than
    $$ \frac{n^{D} - 1}{n - 1} = \frac{n^{\lfloor \delta^{-1}\, \|b - a\|_1 \rfloor + 1} - 1}{n - 1} $$
    iterations,
    returning a $\delta$-relaxed $\epsilon$-optimal solution, or proving that $\mathcal{F}_\delta = \emptyset$.
\end{theorem}
\begin{proof}
    By Lemma~\ref{lem:poa_properties}(i), the returned solution is correct whenever the algorithm terminates, so it is sufficient to bound the number of iterations.

    No node of the vertex tree serves as the maximal vertex at two different iterations.
    At every non-final iteration $t$ the maximal vertex is refined, $\hat{x}^t \in I^t_\delta$ by Lemma~\ref{lem:poa_properties}(ii), and refined vertices never return to the active set, $I^t_\delta \cap V^\tau = \emptyset$ for all $\tau > t$, while the final iteration draws its maximal vertex from what remains of that set.
    Each iteration therefore consumes a distinct node, and the number of iterations is bounded by the number of nodes the tree can hold.
    By Lemma~\ref{lem:vertex_paths}(ii) the tree is at most $D$ nodes deep, and since each node has at most $n$ children, its $d$-th level holds at most $n^{d-1}$ nodes.
    Summing over the levels leaves at most
    $$ \sum\nolimits_{d = 1}^{D} n^{d - 1} = \frac{n^D - 1}{n - 1} $$
    nodes available to act as maximal vertices, which is the stated bound.
\end{proof}

The bound we obtain is very conservative, as it counts every vertex the tree could conceivably contain while the algorithm typically expands only a small fraction of them.
Had we retained the full refinement \eqref{eq:refined_vertices} in Algorithm~\ref{alg:poa_v2}, keeping the balanced anchors and the $\delta$-relaxation but refining every infeasible vertex, we could instead have argued by volume.
Each iteration then removes the whole region $(z^t, \infty)$ from the polyblock, so no later maximal vertex can lie in it.
By Lemma~\ref{lem:poa_properties}(ii) this region contains the box $(\hat{x}^t - \delta\, \unitvector, \hat{x}^t]$ at every non-final iteration.
These boxes are therefore pairwise disjoint, and as each has volume $\delta^n$ and lies in $[a - \delta\, \unitvector, b]$, at most $\lfloor \delta^{-n} \prod_i(b_i-a_i+\delta) \rfloor$ such iterations are possible, plus the last.
Partial refinement gives up this argument, as it leaves part of the infeasible region in place, in return for the guaranteed reduction in coordinate sums on which the bound above rests.
In practice we find that the number of iterations required is much smaller than either bound, particularly when the rectangular set $[a,b]$ is well-scaled.

\subsection{Vectorising POA}\label{ssec:vectorising}

Thus far, each iteration of POA computes only a single projection and explores the corresponding infeasible region.
While this strategy is sufficient for convergence, it does not fully exploit the capabilities of modern parallel computing hardware.
We therefore introduce a vectorised variant of POA which processes multiple candidate solutions concurrently.
This mirrors the node-level parallelism employed by parallel branch-and-bound methods \cite{parallel_bnb_survey}, and in particular by modern MILP solvers, which expand several open nodes of the search tree concurrently in order to exploit multi-threaded hardware \cite{parallel_milp_solvers}.
An important distinction is that branching in MILP partitions a node into disjoint subproblems, so concurrently expanded nodes may be refined independently and merged without interaction.
In POA, by contrast, the infeasible cones generated by different projections may overlap, and the refinement must therefore resolve vertices lying in more than one cone.
We address this below by associating each infeasible vertex with a single projection.

Instead of selecting only the maximal vertex $\hat{x}^t$ at iteration $t$, we select a set of $k$ vertices $\{x^{t,j}\}_{j\in[k]} \subset V^t$ according to an arbitrary selection policy, subject only to the requirement that the maximal vertex is always selected, $\hat{x}^t \in \{x^{t,j}\}_{j\in[k]}$, which ensures the convergence guarantee of Theorem~\ref{thm:POA_new} is preserved.

Each selected vertex is then processed independently, producing its own balanced anchor, projection, and shifted candidate:
\begin{align*}
    y^{t,j} & := x^{t,j} - (\|x^{t,j}-a\|_\infty + \delta)\, \unitvector, \\
    z^{t,j} &:= \pi_{\mathcal{G}_\delta}(x^{t,j}; y^{t,j}), \\
    c^{t,j} &:= \min(z^{t,j} + \delta\,\unitvector,\, b).
\end{align*}
These are the per-vertex analogues of the quantities used in Section~\ref{ssec:relaxed_opt}, and the feasibility argument given there applies unchanged to each $c^{t,j}$.

In principle, the refined search space should be obtained by removing the union of all infeasible cones
$$R^t := \bigcup\nolimits_{j \in [k]}(z^{t,j}, \infty)$$
from the remaining search space.
Computing the vertex representation of $P(V^t) \backslash R^t$, however, appears substantially more involved than in the single-projection case.
Instead, we adopt a simpler refinement strategy in which each vertex is associated with at most one projection.

Following the partial refinement rule of Section~\ref{ssec:relaxed_opt}, a vertex is refined only when it lies at least $\delta$ above one of the projections,
$$I^t_\delta := V^t \cap \bigcup\nolimits_{j\in[k]} (z^{t,j} + \delta\, \unitvector, \infty),$$
with the remaining vertices $F^t_\delta := V^t \backslash I^t_\delta$ carried over unchanged as before.
When more than one cone is eligible to refine $s \in I^t_\delta$, we associate it with the lowest-indexed such projection,
$$ j_s := \min\{j \in [k]: s > z^{t,j} + \delta\, \unitvector\}.$$
After removing redundant vertices, this yields the refined vertex set
\begin{equation}\label{eq:vector_refinement}
    W^t = \{\varrho_i(s, z^{t,j_s}):
    \varrho_i(s, z^{t,j_s}) \notin P( \bar{I}^{t,j_s} \backslash \{s\} ),\,
    s\in I^t_\delta,\, i\in [n] \},
\end{equation}
where
$$\bar{I}^{t,j} := V^t \cap [z^{t,j}, \infty), \quad \forall j \in [k].$$
This replaces \eqref{eq:refined_vertices2} when computing $V^{t+1}$.

The redundancy test is that of \eqref{eq:refined_vertices}, applied to each vertex under the projection it is assigned, and so the refined polyblock still contains the remaining search space,
$$P(F^t_\delta \cup W^t) \supset P(V^t) \backslash R^t \supset \mathcal{F}^{t+1}_\delta.$$

Algorithm~\ref{alg:poa_v3} is the result of combining the three algorithmic improvements proposed: balanced anchors (Section~\ref{ssec:balanced_anchors}), $\delta$-relaxation (Section~\ref{ssec:relaxed_opt}), and vectorisation.
It differs from Algorithm~\ref{alg:poa_v2} only in the body of its main loop, which is all we display.
Both lemmas underlying Theorem~\ref{thm:POA_new} remain valid under this refinement.
Lemma~\ref{lem:poa_properties} applies unchanged, as the selection rule always includes the maximal vertex and its anchor, projection, and shifted candidate are computed exactly as in Algorithm~\ref{alg:poa_v2}, so every non-final iteration still refines $\hat{x}^t$.
Lemma~\ref{lem:vertex_paths} holds equally for the resulting vertex tree, as \eqref{eq:vector_refinement} follows the same partial refinement strategy.
The counting argument of Theorem~\ref{thm:POA_new} uses nothing else, so its iteration bound also holds for Algorithm~\ref{alg:poa_v3}.

Compared with Algorithm~\ref{alg:poa_v2}, the vectorised algorithm exposes substantially more parallelism.
The feasibility-query, projection, and refinement operations for each selected vertex may all be executed independently.
Furthermore, processing multiple projections typically causes more vertices to be refined at each iteration, increasing the amount of parallel work available during the vertex-update stage.
In this work, we exploit this parallelism using CPU multi-threading.
However, in higher-dimensional problems the active vertex set often becomes extremely large, making the proposed vectorisation strategy a promising basis for future GPU implementations capable of processing much larger batches of vertices simultaneously.

\begin{algorithm}
    \caption{The POA algorithm with all proposed improvements. Only the body of the main loop is shown; the surrounding initialisation and termination are as in Algorithm~\ref{alg:poa_v2}, with the additional solver parameter $k$ setting the number of vertices refined per iteration.}
    \label{alg:poa_v3}
    \small\begin{algorithmic}[0]
    \State $\{x^1,\hdots, x^k\} \gets \text{any subset of $V$ containing $\hat{x}$} $
    \For{$j \in [k]$}
        \State $y^j \gets x^j - (\|x^j - a\|_\infty + \delta)\,\unitvector$
        \State $z^j \gets \pi_{\mathcal{G}_\delta}(x^j; y^j)$
        \State $c^j \gets \min(z^j + \delta\,\unitvector,\, b)$
        \If{$f(c^j) > \hat{f}$ and $c^j \in \mathcal{H} \cap [a,b]$}
            \State $\hat{z} \gets c^j$
            \State $\hat{f} \gets f(c^j)$
        \EndIf
    \EndFor
    \State $V \gets$ Update as \eqref{eq:vertex_update} with $W^t$ given by \eqref{eq:vector_refinement}
\end{algorithmic}
\end{algorithm}

\section{Implementing POA algorithms}\label{sec:implementation}
In this section, we develop an efficient implementation of POA applicable to both Algorithm~\ref{alg:poa} and Algorithm~\ref{alg:poa_v3}.
The key idea is to replace the direct storage of polyblock vertices with a tree-based representation of the polyblock construction process.
We first identify the main computational bottlenecks in standard POA implementations, before introducing the proposed data structure and demonstrating how it allows for more efficient implementations of POA's subroutines.

\subsection{Computational bottlenecks in POA}\label{ssec:poa_bottlenecks}

The primary computational challenge in POA arises from the rapid growth of the polyblock vertex set $V^t$ as the problem dimension increases.
During each iteration, POA expands the infeasible vertices $I^t$ by generating up to $n$ new child vertices for each one.
Consequently, the number of vertices can increase exponentially with the number of iterations.
This growth is considerably more severe than in many traditional branch-and-bound algorithms.
For example, branch-and-cut methods for mixed-integer programming typically generate only two child nodes per iteration, whereas a single POA expansion may generate up to $n\cdot|I^t|$ new vertices.

As a result, the dominant computational costs in POA are associated with operations that require searching, updating, or filtering the vertex set $V^t$.
With reference to Section~\ref{ssec:poly_refinement},
each iteration must find the maximal vertex $\hat{x}^t = \argmax_{x \in V^t} f(x)$, answer the range query $\bar{I}^t = \{s \in V^t : s \geq z^t\}$, add the newly generated vertices $N^t$, and remove the vertices discarded by \eqref{eq:vertex_update}: those refined during the iteration, and those left with $f(s) < \hat{f}^{t+1} + \epsilon$.
An efficient POA implementation should store $V^t$ in a data structure which supports all four operations as efficiently as possible.

A straightforward implementation stores the vertices $V^t$ and their corresponding objective values $f(V^t)$ using two dynamic arrays, such as Python lists or C++ vectors.
Addition is then cheap, as new vertices are appended in amortised $O(1)$ time, but the search, the query, and the removal each require a scan of the full arrays, at a cost of $O(|V^t|)$.
Although not explicitly documented, previous implementations of POA (e.g., \cite{monotone_opt_tuy, montone_opt_MISO}) likely employ a similar approach.
We refer to this as the \emph{naive} POA implementation.

While simple and effective for small problems, the naive implementation does not exploit the hierarchical structure inherent in POA.
Every vertex generated during POA is derived from a previous vertex by modifying a single component, meaning that the relationships between vertices contain valuable information which can be used to accelerate these operations.
We therefore propose an alternative representation based on the underlying vertex tree.

\subsection{A tree-based polyblock representation}
Rather than explicitly storing only the current leaf vertices $V^t$, we store the complete tree $\mathcal{T}^t$ generated during the POA search process. We show that this representation provides efficient implementations of all four operations described above while avoiding repeated searches over the complete vertex set.

Maintaining the full tree may appear to increase memory requirements.
However, the structure of POA trees allows a compact representation in which each node stores only the information required to reconstruct its corresponding vertex.
Specifically, we represent $\mathcal{T}^t$ using six equal-length dynamic arrays, with each tree node represented by an index $i$ into them:
\begin{center}
    \addvspace{10pt}
    \small\begin{tabular}{ll} \toprule
        Array & Entry at node $i$ \\ \midrule
        \texttt{start}  & Index of the first child of $i$ \\
        \texttt{end}    & Index immediately after the last child of $i$ \\
        \texttt{parent} & Index of the parent of $i$ \\
        \texttt{comp}   & Vertex component modified at $i$ \\
        \texttt{value}  & Value assigned to the modified component \\
        \texttt{obj}    & Maximum objective value among the descendant leaves of $i$ \\
        \bottomrule
    \end{tabular}
    \addvspace{10pt}
\end{center}

Since each new vertex $\varrho_i(s, z^t) \in N^t$ differs from its parent $s \in V^t$ by only a single component, the complete vertex representation does not need to be stored at every node.
Instead, the \texttt{comp} and \texttt{value} arrays store the sequence of component modifications required to reconstruct each vertex.

Furthermore, because POA only expands leaf nodes and all children of an expanded node are generated within the same iteration, child nodes can be stored using contiguous indices.
Specifically, for a non-leaf node with index $i \in \mathbb{N}_0$, its children occupy the range $[\texttt{start}[i], \texttt{end}[i]).$
This compact representation simplifies tree traversal and enables efficient iteration over child nodes.

We now detail how this tree representation can be used to implement each of the four operations on $V^t$.

\subsubsection*{Finding the maximal vertex}
The \texttt{obj} values stored at each tree node provide an upper bound on the objective values of all descendant leaves.
In particular, the \texttt{obj} value of a node is equal to the maximum objective value among the leaves in its subtree.
Therefore, the tree can be traversed in a manner analogous to a tournament tree, where each internal node stores the best candidate among its children.

The maximal vertex $\hat{x}^t$ can consequently be found by traversing the tree from the root to the leaf with the largest objective value.
Starting from the root node, Algorithm~\ref{alg:max_vertex} maintains the current vertex representation $v$, initially set to $b$.
At each node, the child with the largest \texttt{obj} value is selected, and the corresponding component update is applied to $v$.
Since each child stores only the component that differs from its parent, this process reconstructs the vertex representation of the optimal leaf without explicitly storing all vertices.
\begin{algorithm}
    \caption{A tree-based algorithm for computing the maximal vertex}
    \label{alg:max_vertex}
    \small\begin{algorithmic}[0]
        \State \textbf{Inputs:} Vertex tree: $(\texttt{start},\, \texttt{end},\,\texttt{value},\, \texttt{comp},\, \texttt{obj})$, Problem data: $(b)$
        \State \textbf{Output:}
        Maximal vertex $\hat{x}$, Index of maximal vertex $i$

        \State $\hat{x} \gets b$
        \State $i \gets 0$
        \While{$\texttt{start}[i] < \texttt{end}[i]$}
            \State $i \gets \argmax \{ \texttt{obj}[j] : \texttt{start}[i] \leq j < \texttt{end}[i] \} $
            \State $\hat{x}[\texttt{comp}[i]] \gets \texttt{value}[i]$
        \EndWhile
        \Return $\hat{x}, \, i$
    \end{algorithmic}
\end{algorithm}

\subsubsection*{Range queries}
The hierarchical structure of $\mathcal{T}^t$ also enables efficient range queries over the polyblock.
For any node $s \in T^t$, the rectangular region represented by the leaves of its subtree is contained within $P(s)$.
Consequently, the tree provides a hierarchy of nested regions, similar in spirit to spatial indexing structures such as B-trees \cite{B-trees} and KD-trees \cite{kd_trees_tutorial}, allowing large portions of the search space to be discarded without examining individual vertices.

To identify the set of vertices $\bar{I}^t = \{s \geq z^t : s \in V^t\}$,
we recursively traverse the tree from the root and query whether each child node can contain vertices satisfying the range condition.
Since a child node differs from its parent in only one component, this test can be performed efficiently as $\texttt{value}[i] \geq z^t_{\texttt{comp}[i]},$ where $i$ is the corresponding node index.
Furthermore, the stored \texttt{obj} values allow subtrees with insufficient objective values to be pruned immediately, avoiding unnecessary traversal.

As in Algorithm~\ref{alg:max_vertex}, the full vertex representation of each leaf can be reconstructed during traversal by maintaining the sequence of component updates.
Algorithm~\ref{alg:range_q} gives the complete procedure.

The query returns $\bar{I}^t$, which is used directly to discard redundant vertices in \eqref{eq:refined_vertices} and \eqref{eq:refined_vertices2}.
The vertices to be refined are nested within it,
$$I^t_\delta \subset I^t \subset \bar{I}^t,$$
so each is recovered by a single pass over the queried vertices, comparing them against $z^t$ or $z^t + \delta\, \unitvector$ as required.
Filtering linearly in this way is inexpensive because $\bar{I}^t$ is typically far smaller than the active vertex set $V^t$, which the query never enumerates.
The vectorised algorithm of Section~\ref{ssec:vectorising} concurrently performs one such query per projection, computing $I^t_\delta$ and the assignment $j_s$ from the $k$ results.

\begin{algorithm}
    \caption{A tree-based algorithm for performing range queries}
    \label{alg:range_q}
    \small\begin{algorithmic}[0]
            \State \textbf{Inputs:} Vertex tree: $(\texttt{start},\, \texttt{end},\, \texttt{value},\, \texttt{comp},\, \texttt{obj})$,
            Problem data: $(b,\, z,\, \hat{f},\, \epsilon)$
            \State \textbf{Output:}
            Queried vertices $\bar{\texttt{I}}$,
            Indices of queried vertices $\bar{\texttt{J}}$

        \State $ \texttt{S} \gets [b], \ \texttt{K} \gets [0] $ \Comment{Vertices and indices of nodes awaiting traversal}
        \State $\bar{\texttt{I}} \gets [],\ \phantom{b} \bar{\texttt{J}} \gets []$ \Comment{Range query results}
        \While{$\texttt{S}$ is not empty}
            \State $v \gets \texttt{S} \texttt{.pop()}, \ i \gets \texttt{K} \texttt{.pop()}$
            \If{\texttt{start}[i] = \texttt{end}[i]} \Comment{Leaves are returned, internal nodes descended}
                \State $\bar{\texttt{I}} \texttt{.append(v)}, \ \bar{\texttt{J}} \texttt{.append(i)}$
            \Else
            \ForAll{$j \in \{\texttt{start}[i],\hdots, \texttt{end}[i] - 1\}$}
                \If{$\texttt{value}[j] \geq z_{\texttt{comp}[j]}$ and $\texttt{obj}[j] \geq \hat{f} + \epsilon $}
                    \State $w \gets v$
                    \State $w[\texttt{comp}[j]] \gets \texttt{value}[j]$
                    \State $\texttt{S.append(w)}, \ \texttt{K.append(j)}$
                \EndIf
            \EndFor
            \EndIf
        \EndWhile
        \Return $\bar{\texttt{I}}, \bar{\texttt{J}}$
    \end{algorithmic}
\end{algorithm}

\subsubsection*{Adding vertices}

When new leaf vertices $N^t$ are generated, each of the six arrays defining $\mathcal{T}^t$ is extended with the corresponding attributes of the new nodes, ordered so that siblings occupy contiguous ranges in memory and the traversals described above remain efficient.
Newly created leaves are assigned empty child ranges $\texttt{start}[i] = \texttt{end}[i]$, while the parent nodes expanded during the current iteration are updated to reference their newly generated children.

For explored nodes that do not receive new children, the values of $\texttt{start}$ and $\texttt{end}$ are set to indicate an empty child range, and their \texttt{obj} values are set to $-\infty$.
This prevents them from being selected in the maximal vertex and range searches.
Finally, the new \texttt{obj} values are propagated upwards through the tree, updating each ancestor to match the maximum objective value among its children.

\subsubsection*{Removing vertices}
Two kinds of vertices leave the active set at each iteration: those refined during it, and those left with an objective value below $\hat{f} + \epsilon$.
Neither needs to be removed from the tree structure explicitly.
A refined vertex ceases to be a leaf as soon as its children are attached, so the searches above no longer reach it.
When every one of its children is discarded, it is instead given $\texttt{obj} = -\infty$.
This is the same marking carried by a pruned leaf, whose own objective value has fallen below $\hat{f} + \epsilon$.
Both are therefore excluded by the range query of Algorithm~\ref{alg:range_q}, whose \texttt{obj} test skips them and, through the bounds held at internal nodes, any subtree consisting only of such nodes.
Removal therefore costs nothing beyond the upward propagation of \texttt{obj} already performed when adding vertices.

However, retaining these nodes indefinitely can lead to unnecessary memory consumption as the tree grows.
To address this issue, we periodically rebuild the tree, discarding every node whose \texttt{obj} value has fallen below $\hat{f} + \epsilon$, and with it every refined vertex left without active descendants.
The resulting tree compaction preserves the correctness of POA while limiting long-term memory growth.

\subsection{Complexity of tree-based POA}\label{ssec:tree_complexity}
Although the proposed data structure substantially improves the practical performance of POA, analysing its computational complexity is challenging because the structure of the vertex tree $\mathcal{T}^t$ depends on the underlying monotonic problem.
In particular, tree search procedures are not guaranteed to run in $O(\log |V^t|)$ time, since $\mathcal{T}^t$ is generally unbalanced.
Furthermore, long chains of single-child nodes prevent a simple complexity bound in terms of the number of active vertices $|V^t|$.

For the relaxed algorithms of Section~\ref{ssec:relaxed_opt}, however, Lemma~\ref{lem:vertex_paths}(ii) bounds the depth of the tree by $D = \lfloor \delta^{-1}\|b-a\|_1 \rfloor + 1$, independently of $|V^t|$ and indeed of $t$.
Each individual traversal of the tree therefore has a worst-case cost which is likewise independent of the size of the vertex set.
Algorithm~\ref{alg:max_vertex} descends a single path from the root, examining at most $n$ children at each of its at most $D$ nodes, and so runs in $O(nD)$ time; propagating updated \texttt{obj} values from a refined vertex back to the root traverses such a path in the opposite direction, at the same cost.
The array-based implementation instead scans the entire vertex set for both operations, at a cost of $O(|V^t|)$, which grows rapidly with the problem dimension as discussed in Section~\ref{ssec:poa_bottlenecks}.
Range queries admit no comparable guarantee, as any implementation must at least enumerate the vertices it returns; the \texttt{obj} bounds nonetheless allow whole subtrees containing no active vertices to be skipped.
Adding the new vertices combines the two costs: appending the $|N^t|$ new leaves takes $O(|N^t|)$ time, while every refined vertex has its \texttt{obj} value changed, so the upward pass is repeated once for each element of $I^t_\delta$, at a total cost of $O(|N^t| + |I^t_\delta|\, nD)$.
Table~\ref{tab:operations} collects the resulting costs of the four operations of Section~\ref{ssec:poa_bottlenecks}.

\begin{table}[h]
    \centering
    \caption{Worst-case cost per iteration of each operation on the vertex set, under the two representations. Here $D$ bounds the depth of the vertex tree, $\bar{I}^t$ denotes the result of the range query, and $I^t_\delta$ the set of vertices refined during the iteration.}
    \label{tab:operations}
    \small\begin{tabular}{lll} \toprule
        Operation & Array-based & Tree-based \\ \midrule
        Finding the maximal vertex $\hat{x}^t$  & $O(|V^t|)$   & $O(nD)$ \\
        Answering the range query $\bar{I}^t$   & $O(|V^t|)$   & $\Omega(|\bar{I}^t|)$, output-sensitive \\
        Adding the new vertices $N^t$           & $O(|N^t|)$   & $O(|N^t| + |I^t_\delta|\, nD)$ \\
        Removing the pruned vertices            & $O(|V^t|)$   & Deferred to periodic rebuilds \\
        \bottomrule
    \end{tabular}
\end{table}

This worst case is pessimistic, as it presumes a path in which every refinement removes the least possible coordinate sum.
In practice the vertex trees generated by POA are also much more balanced than the worst case would suggest. In particular, descendant leaves of a given node tend to have similar depths, for which Lemma~\ref{lem:vertex_paths} again provides some intuition.

Consider a non-leaf node $v \in T^t$ and two paths descending from $v$ to leaves of its subtree, one substantially longer than the other.
The lemma implies that the leaf at the end of the longer path must have a significantly smaller coordinate sum.
Since $f$ is increasing, this also tends to yield a smaller objective value, reducing the likelihood that this leaf is maximal.
Moreover, the smaller coordinate values make the leaf less likely to lie within one of the infeasible cones $(z^t,\infty)$.
Consequently, shorter paths are more likely to be extended than longer ones, narrowing the disparity in path lengths over time.
The refinement process therefore exhibits a natural tendency to rebalance the tree.

A rigorous average-case analysis of this self-balancing behaviour remains an interesting direction for future work. In Section~\ref{sec:results}, we instead demonstrate empirically that the proposed representation provides substantial speed-ups over the naive array-based implementation.

\section{\href{https://github.com/RashwanA/polyblocks}{\texttt{polyblocks}}: an open-source implementation of POA}\label{sec:software}
Realising the preceding sections as software requires a solver in which the anchor rule, the relaxation, and the vertex representation can each be varied independently.
As noted in the introduction, no such solver is currently available.
To improve the accessibility and reproducibility of research in this area, we introduce \texttt{polyblocks}, an open-source Python package implementing the algorithms proposed in this paper.

The package serves two complementary purposes.
First, it provides efficient reference implementations of both the naive and tree-based variants of Algorithm~\ref{alg:poa_v3}.
Second, it provides a modular framework for developing and evaluating new POA variants.
Key algorithmic components -- including the policy used for selecting vertices $x^{t,j}$ and their corresponding anchors $y^{t,j}$ -- can be modified independently, allowing researchers to prototype and benchmark new algorithmic ideas with minimal changes to the core implementation.

The package is implemented in pure Python using \texttt{NumPy} \cite{numpy_package} together with \texttt{Numba} \cite{numba_package}, a just-in-time compiler for a subset of the Python language and \texttt{NumPy} functions.
This approach combines the readability and flexibility of Python with performance comparable to implementations based on compiled language extensions \cite{numba_performance_study}, making the package suitable for both research and practical experimentation.

As a demonstration, Listing~\ref{code:poly_example} shows how \texttt{polyblocks} can be used to solve the monotonic optimisation problem
\begin{align}~\label{eq:poly_example}
\begin{split}
    \max_{x \in \mathbb{R}^2 } \quad & x^2_1 + x^2_2  \\
    \ \text{s.t.} \quad &  x_1 + x_2 \leq 1, \\
    & 0 \leq x_1 \leq 1, \\
    &  0 \leq x_2 \leq 0.5
\end{split}
\end{align}
using the tree-based implementation of Algorithm~\ref{alg:poa_v3}.
\begin{listing}[h]
    \begin{lstlisting}
    from polyblocks import TreePOA

    result = TreePOA.solve(
        obj=lambda x: (x**2).sum(1),          # objective
        ub_oracle=lambda x: x.sum(1) <= 1.0,  # normal set oracle
        x_l=(0., 0.),
        x_u=(1., 0.5),
    )

    print(result.x)    # (*@$\approx [1.0, 0.0]$@*)
    print(result.obj)  # (*@$\approx 1.0$@*)
    \end{lstlisting}
    \caption{A code example using the \texttt{polyblocks} package to solve problem \eqref{eq:poly_example}}
    \label{code:poly_example}
\end{listing}

Observe that the objective and oracle functions are assumed to accept batches of vertices, allowing POA's subroutines to exploit vectorised computation.
The complete source code, documentation, and additional examples are available in the project's online repository.

\section{Numerical results}\label{sec:results}
Using \texttt{polyblocks}, we now measure how much each of the proposed improvements contributes in practice, both individually and in combination.

\paragraph{Algorithms}
To isolate the effect of each proposed improvement, we compare five POA variants forming a cumulative ablation: each variant adds a single improvement to the one preceding it, so that \texttt{Base} contains none of our contributions while \texttt{Vectorised} contains all of them.
Table~\ref{tab:variants} summarises the resulting configurations.

\begin{table}[h]
    \centering
    \caption{The five POA variants considered. Each variant adds one improvement to the one above it.}
    \label{tab:variants}
    \small\begin{tabular}{llcccc} \toprule
        \multirow{2}*{Variant} & \multirow{2}*{Algorithm} &
        Balanced & $\delta$-relaxed & Tree-based & Vectorised \\
        & & anchors & optimality & implementation & refinement \\
        & & (\S\ref{ssec:balanced_anchors}) & (\S\ref{ssec:relaxed_opt}) &
        (\S\ref{sec:implementation}) & (\S\ref{ssec:vectorising}) \\ \midrule
        \texttt{Base}       & Alg.~\ref{alg:poa}     &            &            &            &            \\
        \texttt{Balanced}   & Alg.~\ref{alg:poa}     & \checkmark &            &            &            \\
        \texttt{Relaxed}    & Alg.~\ref{alg:poa_v2}  & \checkmark & \checkmark &            &            \\
        \texttt{TreeBased} & Alg.~\ref{alg:poa_v2}  & \checkmark & \checkmark & \checkmark &            \\
        \texttt{Vectorised} & Alg.~\ref{alg:poa_v3}  & \checkmark & \checkmark & \checkmark & \checkmark \\
        \bottomrule
    \end{tabular}
\end{table}

The first two variants differ only in their choice of projection anchor, using the fixed anchors \eqref{eq:old_anchor} with $\rho = 0.2$, and the balanced anchors \eqref{eq:new_anchors} respectively.
The first three use the naive implementation described in Section~\ref{ssec:poa_bottlenecks}, while the remaining two use the tree-based implementation of Section~\ref{sec:implementation}.
The \texttt{Vectorised} variant refines $k=8$ vertices per iteration.

\paragraph{Test problems}
We evaluate these algorithms on three classes of monotonic problems, all of the form
$$\max_{x \in [0,1]^n} f(x),
\quad \text{s.t.} \ g(x) \leq 0, \ h(x)\geq 0,$$
where each of $f$, $g$ and $h$ is drawn from the same class: a quadratic $x^TQx + q^Tx$, a monotone neural network $w^T(Vx + v)^+$, or a step function obtained by rounding such a network to the nearest $0.02$.
All parameters are non-negative, $Q \in \mathbb{R}^{n\times n}_+$, $q \in \mathbb{R}^{n}_+$, $V \in \mathbb{R}^{m\times n}_+$ and $w \in \mathbb{R}^{m}_+$, apart from the bias $v \in \mathbb{R}^{m}$, which leaves every such function increasing and the resulting problems monotonic.
Rounding preserves monotonicity while making the function discontinuous.
We consider problems in 4, 5, and 6 dimensions.
For each dimension and class, we evaluate the algorithms on a fixed set of 20 randomly drawn instances.

\paragraph{Experimental setup}
We ran all experiments on a single machine with 16GB RAM and an AMD Ryzen 5 3600 processor.
All algorithms use a relative value of $\epsilon$ equal to $0.01$ and, where applicable, a relaxation value of $\delta = 0.001$.
For each problem instance, we limited each solver to run for a maximum of an hour and for at most $10^8$ POA iterations.
To control memory usage and runtime, we limited tree-based implementations to at most $2 \times 10^8$ tree nodes $|T^t|$ and naive implementations to at most $3 \times 10^6$ vertices $|V^t|$.
We used a reduced limit for naive implementations as they were found to stall much earlier than their tree-based counterparts.
A solver fails to find a solution if it exceeds the maximum runtime, iteration, or memory limits.

\begin{figure}[t]
\centering
\includegraphics[width=\textwidth]{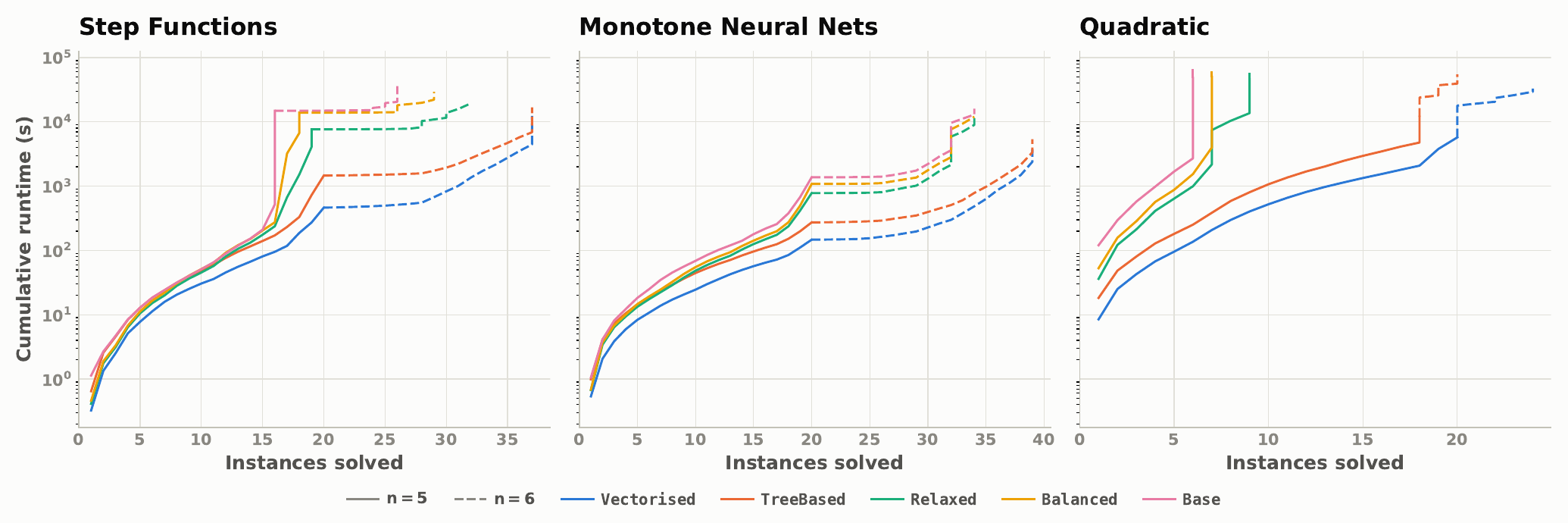}
\caption{
Cumulative runtime vs problem instances solved for each solver variant and each problem class considered.
Runtime includes time spent on problems which a solver fails on.
We consider both five- and six-dimensional instances for each problem class, as shown by the line style of each curve.
Four-dimensional instances are omitted for visual clarity.
}
\label{fig:cactus_plot}
\end{figure}

\paragraph{Discussion}
Figure~\ref{fig:cactus_plot} shows one cactus plot per problem class, giving cumulative runtime against the number of instances solved.
Instances are sorted by runtime along each curve, and those a solver fails on still contribute their runtime to the cumulative total.
The curves are ordered by the ablation on all three classes: each variant reaches any given number of solved instances in less time than the one above it in Table~\ref{tab:variants}, and goes on to solve more of them.
On the step function class, for instance, \texttt{Vectorised} reaches 20 solved instances roughly $30 \times$ faster than \texttt{Base}.

Table~\ref{tab:geometric_means} reports shifted geometric mean runtimes for each problem class, solver, and dimension considered, along with the number of instances successfully solved.
Unlike the cactus plots, the mean for each solver is taken only over the instances it solved, so a solver failing on the hardest instances is scored on the easier ones alone.
The counts should therefore be read first: they increase monotonically along the ablation in every row, most sharply for the five-dimensional quadratics, where they rise from 6 solved to 20.
Where the counts agree the runtimes are directly comparable and improve monotonically as well; where they differ widely, as on the six-dimensional step functions, the weaker solvers' means are depressed by the easy subsets on which they were scored.
The table also isolates the effect of dimension, with runtimes growing exponentially in $n$ for every variant, in line with the vertex-set growth described in Section~\ref{ssec:poa_bottlenecks}.

Of the four improvements, the tree representation yields the largest gains, followed by vectorisation, both of which act directly on that growth.
Replacing the $O(|V^t|)$ scans of Table~\ref{tab:operations} with $O(nD)$ traversals pays off in proportion to the active vertex set: the tree is marginally slower on two of the four-dimensional instances, but advantageous on the five-dimensional quadratics, where it raises the number solved from 9 to 18.
The $\delta$-relaxation curbs the same growth through its partial refinement rule, leaving fewer vertices to be stored and searched.
Its effect is clearest on the neural network instances, where \texttt{Relaxed} and \texttt{Balanced} solve equally many instances at $n = 5$ and $n = 6$, and the relaxation reduces their runtimes by roughly $15\%$ in both cases.
Its larger contribution, however, is to reach rather than to speed: relaxation alone bounds the number of iterations, guaranteeing termination through Theorem~\ref{thm:POA_new}.
The step function class is included partly for this reason, being discontinuous and so outside the guarantee of \cite{sit_algorithm}, while Theorem~\ref{thm:POA_new} applies to it unchanged.
Little is given up in exchange: the relaxed variants solve \eqref{eq:monotone-canon} only to within $\delta$, but on commonly solved instances the shortfall in objective value was negligible compared to the objective tolerance $\epsilon$.

\begin{table}[h]
\centering
\caption{1\,second-shifted geometric mean runtime (s) of solved instances, with the number of instances solved, out of a total of 20, in parentheses.}
\label{tab:geometric_means}
\small\setlength{\tabcolsep}{4.5pt}\begin{tabular}{l l r r r r r }
\toprule
 Problem class & $n$ & \texttt{Vectorised} & \texttt{TreeBased} & \texttt{Relaxed} & \texttt{Balanced} & \texttt{Base} \\
\midrule
Step functions & 4 & 1.2 (20) & 1.8 (20) & 1.3 (20) & 1.5 (20) & 1.7 (20) \\
 & 5 & 8.3 (20) & 15.3 (20) & 18.4 (19) & 18.5 (18) & 10.8 (16) \\
 & 6 & 67.1 (17) & 91.1 (17) & 73.2 (13) & 57.7 (11) & 53.8 (10) \\
\midrule
Monotone Neural Nets & 4 & 0.5 (20) & 0.8 (20) & 0.6 (20) & 0.6 (20) & 0.7 (20) \\
 & 5 & 5.0 (20) & 8.6 (20) & 12.4 (20) & 14.3 (20) & 18.3 (20) \\
 & 6 & 25.8 (19) & 35.9 (19) & 44.0 (14) & 53.1 (14) & 60.9 (14) \\
\midrule
Quadratic & 4 & 6.4 (20) & 9.3 (20) & 10.4 (20) & 13.6 (20) & 18.5 (20) \\
 & 5 & 107.2 (20) & 165.2 (18) & 335.8 (9) & 269.4 (7) & 346.9 (6) \\
 & 6 & 1954.0 (4) & 1970.2 (2) & N/A (0) & N/A (0) & N/A (0) \\
\bottomrule
\end{tabular}
\end{table}

\section{Conclusions}
We have presented a set of algorithmic and implementation improvements addressing the scalability limitations of the polyblock outer-approx\-imation algorithm.

On the algorithmic side, we first generalised POA to admit arbitrary projection anchors, and showed that this additional freedom naturally yields a balanced monotonicity cut which maximises the volume guaranteed to be removed from the search space at each iteration.
We then introduced a $\delta$-relaxed optimality criterion which restores finite termination with an explicit iteration bound (Theorem~\ref{thm:POA_new}), without requiring the continuity assumptions on which previous robust variants rely.
Finally, we proposed a vectorised variant which refines several vertices per iteration, exposing substantially more parallelism than the sequential procedure while retaining the same convergence guarantee.

On the implementation side, we observed that the vertices maintained by POA carry a tree structure induced by the refinement process itself.
Storing this tree, rather than the vertex set alone, allows each of POA's core subroutines to be implemented as a tree traversal instead of a linear scan.
Our numerical experiments confirm that each proposed improvement meaningfully contributes to the overall speed-up, with the largest gains appearing in higher dimensions where the active vertex set is largest.
We additionally introduced \texttt{polyblocks}, an open-source implementation of these algorithms which we hope will lower the barrier to further research in monotonic optimisation.

Several directions remain open.
The iteration bound of Theorem~\ref{thm:POA_new}, while explicit, counts every vertex the refinement tree could conceivably contain and is far more pessimistic than the behaviour observed in practice; establishing tighter bounds under additional structural assumptions would be valuable.
Similarly, the vertex trees generated by POA are empirically far better balanced than the worst case suggests, and a rigorous average-case analysis of the self-balancing behaviour discussed in Section~\ref{ssec:tree_complexity} remains open.
Finally, our vectorised algorithm is currently parallelised using CPU multi-threading.
Since the active vertex set grows rapidly with problem dimension, a GPU implementation capable of processing much larger batches of vertices concurrently is a natural next step.

\section*{Acknowledgements}
We would like to thank Keith Briggs for his suggestions on improving the software presented, and Chris Budd for his feedback on the manuscript.
This research was supported by the EPSRC Programme Grant EP/V026259/1.

\printbibliography
\end{document}